\documentclass[11pt]{article}

\usepackage[english]{babel}
\usepackage[T1]{fontenc}
\usepackage[normalem]{ulem}

\usepackage[lastexercise,answerdelayed]{exercise}
\usepackage{lipsum}
\usepackage{xcolor} 
\usepackage{amssymb, amsfonts, amsmath, amsthm} 
\usepackage{bigints} 
\usepackage{braket} 
\usepackage{physics} 
\usepackage{enumerate} 
\usepackage{caption}
\usepackage{tikz-cd}
\usepackage{tikz} 
\usetikzlibrary{chains, positioning, shapes.symbols,shapes,arrows}
\definecolor{darkblue}{RGB}{9,72,90}
\definecolor{lightblue}{RGB}{124,184,201}
\usepackage{pgf} 
\usetikzlibrary{angles,quotes}
\DeclareCaptionType{equ}[][]
\usepackage{subfiles} 

\usepackage[ruled,vlined]{algorithm2e}
\usepackage{pgfplots}
\pgfplotsset{compat=1.16}
\usepackage{amsmath,amssymb,amsfonts}
\usepackage{mathtools, nccmath}
\usepackage{mathrsfs} 

\usepackage[square, numbers, comma, sort&compress]{natbib}

\usepackage{circuitikz}

\usepackage{multicol} 

\usepackage[nottoc]{tocbibind} 
\usepackage[utf8]{inputenc}
\usepackage[top=1 in,bottom=1in, left=1 in, right=1 in]{geometry}
\usepackage{hyperref}

\usetikzlibrary{graphs,graphs.standard,quotes}

\usepackage{listings}
\newtheorem{theorem}{Theorem}
\newtheorem{proposition}{Proposition}
\newtheorem{lemma}{Lemma}

\newtheorem{definition}{Definition}

\newtheorem{remark}{Remark}

\usepackage{stmaryrd}

\newtheorem{assumption}{Assumption}

\title{Brownian motion on the Fubini extension space and applications}

\author{%
  Hamed Amini\thanks{University of Florida, USA. \href{mailto:aminil@ufl.edu}{aminil@ufl.edu}} \and
  Nina H.~Amini\thanks{CNRS, L2S, CentraleSupélec, Université Paris-Saclay, France.
  \href{mailto:nina.amini@centralesupelec.fr}{nina.amini@centralesupelec.fr}} \and
  Sofiane Chalal\thanks{L2S, CentraleSupélec, Université Paris-Saclay, France.
  \href{mailto:sofiane.chalal@centralesupelec.fr}{sofiane.chalal@centralesupelec.fr}} \and
  Gaoyue Guo\thanks{MICS, CentraleSupélec, Université Paris-Saclay, France.
  \href{mailto:gaoyue.guo@centralesupelec.fr}{gaoyue.guo@centralesupelec.fr}}%
}
\date{\today}
\begin{document}
\maketitle

{\begin{abstract}
We study a family of essentially pairwise independent Brownian motions indexed by a continuum of labels and show how the Fubini extension framework provides a rigorous way to represent such families as a single jointly measurable process. Within this framework, we address two main objectives: first, we show how a system of graphon stochastic differential equations can be reformulated as a single McKean-Vlasov type equation driven by a standard Brownian motion, which significantly facilitates its analysis. Second, we establish a Girsanov theorem for a continuum of essentially pairwise independent Brownian motions.



\end{abstract}}

\section{Introduction}
Modeling large populations of interacting stochastic particles often requires an idealization in which the set of particles is treated as a continuum \cite{aumann64markets,touboul14propagation}.
From a practical point of view, an atomless probability index space
provides a convenient idealization for many models with a large but finite number of agents. From a technical point of view, such an idealization is often necessary to treat some models of heterogeneous interacting particle systems in the mean-field regime.

In such situations where one needs to model a continuum of agents, the usual product probability space
$(I\times\Omega,\mathcal{I}\otimes\mathcal{F},\lambda\otimes\mathbb{Q})$
cannot accommodate a jointly measurable continuum of independent random variables \cite[Proposition 1.1]{sun98almost}, \cite{judd85law}.
Thanks to the Fubini extension framework introduced by Sun \cite{sun06exact}, one can enlarge this product space to a larger one in which joint measurability holds. Moreover, this enlargement makes it possible to integrate successively with respect to $u$ and $\omega$ in any order, that is, it preserves the Fubini property.

The purpose of this note is to highlight an advantage of working directly on
the Fubini extension of the product space. In this setting, a collection of
random variables can be viewed as a single random variable defined on the
Fubini space. In particular, in the case of games with a continuum of agents \cite[Section 3.7]{carmonadelarue18I}, 
\cite{aurell22sto,amini23graphon}, the dynamics is typically described by a system of
stochastic differential equations driven by a continuum of independent Brownian motions. We show that such a family admits a natural realization as a standard Brownian motion on the Fubini extension space. This representation allows various systems of stochastic differential equations studied in the literature \cite{coppini25nonlinear,amini25gqfs,bayraktar23graphon,crucianelli24interacting} to be viewed as a single McKean-Vlasov type equation on the extension space.  This reformulation simplifies the analysis of well-posedness and provides a rigorous foundation for efficient numerical schemes based on random sampling of the index space. Moreover,
we establish a Girsanov theorem for a collection of essentially pairwise
independent Brownian motions, which makes it possible to construct a
Brownian motion on a new space, where even though the whole process is Brownian, its components are no longer
Brownian individually.

The remainder of the paper is organized as follows. In Section~\ref{sec:Fubini} we recall the Fubini extension framework. Section~\ref{sec:processes} develops processes on
the Fubini extension space and shows that a collection of Brownian motions constitutes a
standard Brownian motion on the extension (Theorem~\ref{thmwienerfubini}).
Section~\ref{sec:application} discusses applications of viewing the system as a single process: we
recast graphon SDEs as a single McKean-Vlasov type equation, and we establish a
Girsanov theorem in Section~\ref{sec:girsanov} for a collection of pairwise independent Brownian motions on the Fubini extension
space and discuss its consequences.

 \paragraph*{Notation.}
Fix a finite horizon $T>0$. For a separable Banach space $E$, set
$\mathcal{C}_E:=\mathcal {C}([0,T];E)$
for the space of continuous functions from $[0,T]$ to $E$, endowed with the
topology of uniform convergence.

For a measure space $(S,\Sigma,\mu)$ and a seperable Banach space $E$, define
$$
L^2_{E}(S,\mu)
:=\Big\{ X:S\to E \;\text{is }\Sigma\text{--}\mathfrak{B}(E)\text{-measurable and }
\int_S \|X(s)\|_E^2\,\mu(\mathrm{d}s)<\infty \Big\}.
$$
As usual, we identify elements of $L^2_{E}(S,\mu)$
that are equal $\mu-$almost everywhere (a.e.).
Let $I := [0,1]$ with Borel $\sigma$–algebra $\mathfrak{B}_I$ and Lebesgue measure
$\lambda_I$ on $(I,\mathfrak{B}_I)$.

Throughout the paper, $C$ denotes a positive constant whose value may change
from line to line.

\section{Continuum of random variables and Fubini extension}\label{sec:Fubini}

We begin by recalling a very important notion due to Sun \cite{sun98almost} that shows that all the notions of independence are, in fact, almost identical to their pairwise counterparts in an ideal setting of a continuum of random variables.

\begin{definition}[Essential pairwise independence]
Let $(I', \mathcal{I}', \lambda')$ and $(\Omega', \mathcal{F}', \mathbb{Q}')$ be two probability spaces. A random variable $\Theta : I' \times \Omega' \to \mathbb{R}$, is said to be \emph{essentially pairwise independent} (e.p.i.) if for $\lambda'$-almost every $u \in I'$, the random variables $\Theta(u,.)$ and $\Theta(v,.)$ are independent for $\lambda'$-almost every $v \in I'$. 
\end{definition}

We now turn to the notion of Fubini extension, which provides the appropriate
probability space to host a jointly measurable e.p.i. family. This extension
enlarges the classical product space in such a way that both independence and
Fubini’s theorem are preserved, and will serve as the basic framework for all
the results in this paper. 
 Before doing so, recall that  an extension $\boldsymbol{\pi}$ of a probability space $(S,\Sigma,\mu)$ to a  new probability space $(S',\Sigma',\mu')$
\begin{align*}
    \boldsymbol{\pi} : (S',\Sigma',\mu') \mapsto (S,\Sigma,\mu)
\end{align*}
is a probability preserving measurable map $\boldsymbol{\pi} : S' \mapsto S$, such that
$$\mu'(\boldsymbol{\pi}^{-1}(B)) = \mu(B), \quad \forall B \in \Sigma.$$

\begin{definition}[Fubini extension] Let $(I',\mathcal{I}',\lambda')$ and
$(\Omega',\mathcal{F}',\mathbb{Q}')$ be two probability spaces.
A \emph{triple} $(I'\times\Omega' ,\mathcal{V}',\mathcal{Q}')$ is called a
Fubini extension of the product space $(I'\times\Omega',\mathcal{I}'\otimes\mathcal{F}',\lambda'\otimes\mathbb{Q}')$ if for any $\mathcal{Q}'$-integrable function $g$ on $(I'\times \Omega',\mathcal{V}')$ real valued:
\begin{itemize}
\item the functions $g_u : \omega \mapsto g(u,\omega)$ and $g_{\omega} : u \mapsto g(u,\omega)$ are integrable on $(\Omega', \mathcal{F}',\mathbb{Q}') $ for $\lambda'-$a.e. $u \in I'$, and on $(I',\mathcal{I}',\lambda')$ for $\mathbb{Q}'$-a.e. $\omega \in \Omega'$, respectively. 
\item the functions $u \mapsto \int_{\Omega'}g_u(\omega)\mathbb{Q}'(\mathrm{d}\omega)$ and $\omega \mapsto \int_{I'}g_{\omega}(u)\lambda'(\mathrm{d}u)$   are integrable, respectively, on $(I',\mathcal{I}',\lambda')$ and  $(\Omega', \mathcal{F}',\mathbb{Q}') $, with 
{{\small} \begin{align*}\int_{I'\times\Omega'}g(u,\omega)\mathcal{Q}'(\mathrm{d}u,\mathrm{d}\omega)  &= \int_{I'}\Bigl(\int_{\Omega'}g_u(\omega)\mathbb{Q}'(\mathrm{d}\omega)\Bigr)\lambda'(\mathrm{d}u)           =\int_{\Omega'}\Bigl(\int_{I'} g_{\omega}(u)\lambda'(\mathrm{d}u)\Bigr)\mathbb{Q}'(\mathrm{d}\omega).
\end{align*}}
\end{itemize}
It is common in the literature to denote the Fubini extension
$(I'\times \Omega', \mathcal{V}', \mathcal{Q}')
$ by $(I' \times \Omega', \mathcal{I} '\boxtimes \mathcal{F}', \lambda' \boxtimes \mathbb{Q}')$ to reflect the fact that the probability space has its marginal space $(I',\mathcal{I}',\lambda')$ and $(\Omega', \mathcal{F}',\mathbb{Q}')$.
\end{definition}

The following theorem from \cite[Theorem 1]{sun09individual}, guarantees the existence of a Fubini-extension space that carries a collection of e.p.i.\ jointly measurable random variables.

\begin{theorem}[\cite{sun09individual}]\label{thm-sun}
Consider $(I,\mathfrak{B}_{I},\lambda_{I})$ to be the Lebesgue index space. Then there exists a probability space $(I,\mathcal{I},\lambda)$ extending $(I,\mathfrak{B}_{I},\lambda_{I})$, a probability space $(\Omega,\mathcal{F},\mathbb{Q})$, and a Fubini extension 
$
(I\times\Omega,\mathcal{V},\mathcal{Q})$ of $
(I\times\Omega,\mathcal{I}\otimes\mathcal{F},\lambda\otimes\mathbb{Q})$
such that for any measurable mapping 
{{\small} \begin{align*}
\varphi : I \to \mathrm{P}(E),
\end{align*}}
where $\mathrm{P}(E)$ denotes the set of Borel probability measures on $E$, there exists a $\mathcal{V}$-measurable random variable,
{{\small} \begin{align*}
\Theta \colon I\times\Omega \to E,
\end{align*}}
such that the random variables $\Theta(u,.)$ are e.p.i. and 
{{\small} \begin{align*}
\mathbb{Q}\circ (\Theta({u},.)^{-1})=\varphi(u),
\quad\text{for all } u\in I.
\end{align*}}
\end{theorem}

{{\small} \begin{center}
\begin{tikzpicture}[
  node distance=2cm,
  every node/.style={inner sep=2pt},
  >=Latex
]
\node (A) at (0,2)   {$\bigl(I,\mathfrak{B}_{\mathcal{I}},\lambda_{I}\bigr)$};
\node[right=0.3cm of A] {Lebesgue index space};

\node (D) at (0,0) {$\bigl(I,\mathcal{I},\lambda\bigr)$};
\node[left=0.3cm of D] {Index space};

\node (B) at (6,0)   {$\bigl(\Omega,\mathcal{F},\mathbb{Q}\bigr)$};
\node[right=0.3cm of B] {Sample space};

\node (E) at (3,-2) {$\bigl(I\times\Omega,\mathcal{I}\otimes\mathcal{F}, \lambda\otimes\mathbb{Q}\bigr)$};
\node[right=0.3cm of E] {Product space};

\node (F) at (3,-4)  {$\bigl(I\times\Omega, \mathcal{V},\mathcal{Q}\bigr)$};
\node[right=0.3cm of F] {Fubini extension space};

\draw[->] (A) -- node[left]{{\small} extend} (D);
\draw[->, dashed] (D) -- (E);
\draw[->, dashed] (B) -- (E);
\draw[->] (E) -- node[left]{{\small} extend} (F);
\end{tikzpicture}

\begin{center}
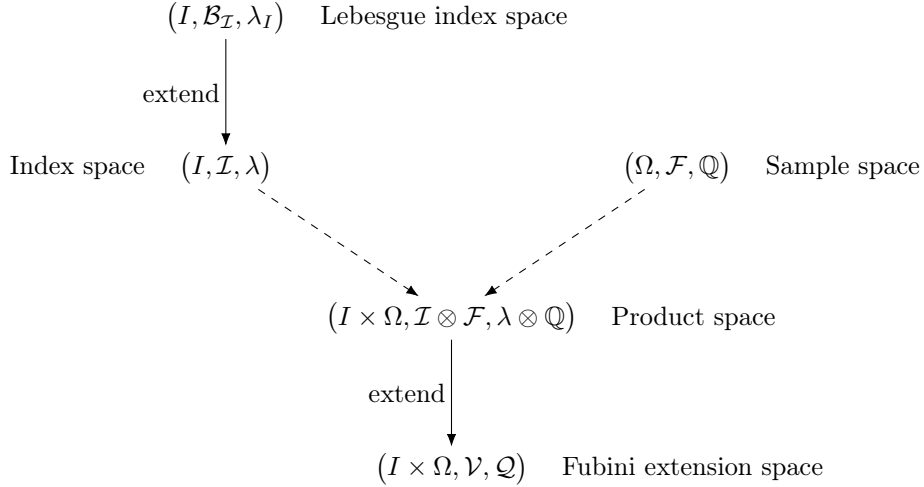
    \textbf{Figure :} Schematic diagram for construction of  Fubini extension probability space. 
\end{center}
\end{center}
}

In the following, we work on 
$(I\times\Omega,\mathcal{V},\mathcal{Q})$, which refers  to the Fubini extension probability space with marginals:
\begin{itemize}
    \item Index space: $(I,\mathcal{I},\lambda)$,
    \item Sample space:
$(\Omega,\mathcal{F},\mathbb{Q})$.
\end{itemize}

\section{Processes on Fubini extension space}\label{sec:processes}

Before starting with processes on the Fubini extension space, let us recall some
basic facts about functions on product spaces.

A function $\Theta:I\times\Omega\to E$ can be viewed in two ways. If we fix
$u\in I$, we obtain a function of $\omega\in\Omega$,
$$
\Theta(u,\cdot): \Omega \to E,\qquad \omega\mapsto \Theta(u,\omega).
$$
For convenience, we denote $\theta_{u}(\omega):=\Theta(u,\omega)$. Thus any
two–variable function $\Theta$ corresponds to an $I$–indexed family of
one–variable functions,
$$
\Theta \equiv \{\theta_{u}:\Omega\to E,\; u\in I\}.
$$

The use of the Fubini extension space ensures the existence of a jointly measurable collection $\{ \theta_u : u \in I\} $. Hence, for the random variable $\Theta $, the advantage of this viewpoint is that we can deal directly with random variable \footnote{This is also valid for stochastic processes, since they can be viewed as random variables taking values in certain Polish path spaces.} $\Theta$ defined on the Fubini extended space $(I\times\Omega, \mathcal{V},\mathcal{Q})$.

\subsection{Brownian motion}
{We first recall the standard definitions of a Brownian motion and an e.p.i. collection of Brownian motions on a probability space.

\begin{definition}[Brownian motion]\label{standardwiener}
A process $(\mathbf{B}_t)_{t \geq 0}$ with $\mathbf{B}_t: I\times\Omega \to \mathbb{R}$
is a standard Brownian motion on $(I\times\Omega,\mathcal V,\mathcal{Q})$ if
\begin{enumerate}
  \item $\mathbf{B}_0=0$, $\mathcal Q$-a.s.;
  \item For all $0\leq s<t$, $\mathbf{B}_t-\mathbf{B}_s\sim \mathcal{N}(0,t-s)$ under $\mathcal Q$;
  \item For any $0\leq t_0<\cdots<t_n$, the increments
        $\mathbf{B}_{t_k}-\mathbf{B}_{t_{k-1}}$ are independent under $\mathcal Q$;
  \item $t \mapsto \mathbf{B}_t(u,\omega)$ is continuous for $\mathcal{Q}$-a.e.\ $(u,\omega)$.
\end{enumerate}
\end{definition}

\begin{definition}[e.p.i. collection of Brownian motion]
A collection $\{ (B^{u}_t)_{t \geq 0} : u \in I\}$ on $(I\times\Omega,\mathcal{V},\mathcal{Q})$ is called an e.p.i. collection of Brownian motion if:
\begin{enumerate}
\item  The map $(u,\omega)\to B^u_t(\omega)$ is jointly measurable for all $t$;
\item The collection is essentially pairwises independents and, for every $u\in I$, the process $(B^{u}_{t})_{t \geq 0},$
is a standard Brownian motion  on $(\Omega,\mathcal{F},\mathbb{Q})$.
\end{enumerate}
\end{definition}

The following theorem shows that collections of e.p.i. Brownian motions admit a natural realization on the Fubini extension.

\begin{theorem}\label{thmwienerfubini}
Let $\{ B^u : u \in I\}$ a collection of e.p.i. Brownian motions on the Fubini extension space $(I\times \Omega,\mathcal{V},\mathcal{Q}).$ Then the process defined by 
$$
\mathbf{B}_t(u,\omega) := B_t^{u}(\omega), \qquad (u,\omega)\in I\times\Omega,
$$
is a standard Brownian motion on the Fubini extension
$(I\times\Omega,\mathcal{V},\mathcal{Q})$.

Reciprocally, let $(\mathbf{B}_t)_{t \geq 0}$ be a standard Brownian motion on the Fubini
extended space $(I\times\Omega,\mathcal{V},\mathcal{Q})$ such that the family $\{ (\mathbf{B}_t(u,.)_{t \geq 0} : \ u\in I\}$ is e.p.i and identically distributed elements. Then the family
forms an e.p.i. collection of Brownian motions.
\end{theorem}

\begin{proof}
We first prove the direct implication. We verify points (1)-(4) of Definition~\ref{standardwiener} for the process $(\mathbf{B}_t)_{t\ge0}$. The first point is straightforward. For the second point, consider a bounded measurable function $g$. Then
{{\small} 
\begin{align*}
    \mathbb{E}_{\mathcal{Q}}[g(\mathbf{B}_t-\mathbf{B}_s)] &= \int_{I}\mathbb{E}_{\mathbb{Q}}[g(B^u_t-B_s^{u})]\lambda(\mathrm{d}u)\\
    &= \int_{I}\int_{\mathbb{R}}g(x)\nu_t(x)\mathrm{d}x\lambda(\mathrm{d}u)= \int_{\mathbb{R}}g(x)\nu_t(x)\mathrm{d}x,
\end{align*}}
where $\nu_t$ is a density of the normal law $\mathcal{N}(0,t-s)$. For the independence of increments,  fix $0\leq t_0<\cdots<t_n$ and bounded measurable function $\varphi_k$. Set
{{\small} \begin{align*}
\mathbf{Z}_k(u,\omega)=\varphi_k\big(\mathbf{B}_{t_k}(u,\omega)-\mathbf{B}_{t_{k-1}}(u,\omega)\big)
=\varphi_k\big(B^u_{t_k}(\omega)-B^u_{t_{k-1}}(\omega)\big).
\end{align*}}
Then,
{{\small} \begin{align*}
\mathbb{E}_{\mathcal{Q}}\Big[\prod_{k=1}^n \mathbf{Z}_k\Big]
&= \int_{I\times\Omega} \prod_{k=1}^n 
\varphi_k\big(B^u_{t_k}(\omega)-B^u_{t_{k-1}}(\omega)\big)
\mathcal{Q}(\mathrm{d}u,\mathrm{d}\omega)\\
&\stackrel{\text{Fubini}}{=} 
\int_I \mathbb{E}_{\mathbb{Q}}\Big[\prod_{k=1}^n 
\varphi_k\big(B^u_{t_k}-B^u_{t_{k-1}}\big)\Big]\lambda(\mathrm{d}u)\\
&\stackrel{\text{indep. incr.}}{=}
\int_I \prod_{k=1}^n \mathbb{E}_{\mathbb{Q}}\Big[
\varphi_k\big(B^u_{t_k}-B^u_{t_{k-1}}\big)\Big]\lambda(\mathrm{d}u)\\
&\stackrel{\text{same law in }u}{=} 
\int_I \prod_{k=1}^n \mathbb{E}_{\mathbb{Q}}\Big[
\varphi_k\big(B_{t_k}^{1}-B_{t_{k-1}}^{1}\big)\Big]\lambda(\mathrm{d}u)\\
&= \prod_{k=1}^n \mathbb{E}_{\mathbb{Q}}\Big[
\varphi_k\big(B^1_{t_k}-B^1_{t_{k-1}}\big)\Big]\\
&= \prod_{k=1}^n \int_I \mathbb{E}_{\mathbb{Q}}\Big[
\varphi_k\big(B^u_{t_k}-B^u_{t_{k-1}}\big)\Big]\lambda(\mathrm{d}u)\\
&= \prod_{k=1}^n \mathbb{E}_{\mathcal{Q}}[\mathbf{Z}_k].
\end{align*}}

\medskip
For point (4) (path continuity), define
{{\small} \begin{align*}
    A = \big\{(u,\omega) : t \mapsto \mathbf{B}_{t}(u,\omega) \; \text{is continuous}  \big\},
\end{align*}}
and define the set
{{\small} \begin{align*}
    A_u  = \big\{\omega : t \mapsto \mathbf{B}_{t}(u,\omega) \; \text{is continuous}  \big\}.
\end{align*}}
Then, for all $(u,\omega) \in I\times\Omega$, we have $\mathbb{I}_{A}(u,\omega) = \mathbb{I}_{A_u}(\omega)$. 
Thus,
{{\small} \begin{align*}
\mathcal{Q}\big(A)=\mathbb{E}_{\mathcal{Q}}[\mathbb{I}_{A}] = \int_{I}\underbrace{\mathbb{E}_{\mathbb{Q}}[\mathbb{I}_{A_u}]}_{\mathbb{Q}(A_u) = 1}\lambda(\mathrm{d}u) = 1.
\end{align*}}
Hence $\mathbf{B}$ is a Brownian motion.

We now prove the reciprocal implication. Fix $v\in I$ and let $g$ be bounded measurable. Then
{{\small}
\begin{align*}
    \mathbb{E}_{\mathcal{Q}}[g(\mathbf{B}_t-\mathbf{B}_s)]
    &= \int_{I}\mathbb{E}_{\mathbb{Q}}[g(B_t^u-B_s^u)]\lambda(\mathrm{d}u) \\
    &= \int_{I}\mathbb{E}_{\mathbb{Q}}[g(B_t^v-B_s^v)]\lambda(\mathrm{d}u) \\
    &= \mathbb{E}_{\mathbb{Q}}[g(B_t^v-B_s^v)], \quad \forall v\in I,
\end{align*}}
where we used the identity in law in the second line. 

For independence of increments we proceed in the same way. For continuity, set
{{\small}
\begin{align*}
A := \{(u,\omega)\in I\times\Omega:\ t\mapsto \mathbf{B}_t(u,\omega)\ \text{is continuous}\}.
\end{align*}}
By Definition~\ref{standardwiener}~(4), we have $\mathcal{Q}(A)=1$. For $u\in I$, define
{{\small}
\begin{align*}
A_u := \{\omega\in\Omega:\ (u,\omega)\in A\}.
\end{align*}}
Then $\mathbb{I}_A(u,\omega)=\mathbb{I}_{A_u}(\omega)$, and by Fubini,
{{\small}
\begin{align*}
\int_I \mathbb{Q}(A_u)\lambda(\mathrm{d}u)
= \mathbb{E}_{\mathcal{Q}}[\mathbb{I}_A]
= \mathcal{Q}(A) = 1.
\end{align*}}
Take $H := \{u\in I:\ \mathbb{Q}(A_u)\leq c<1\}$, then
{{\small}
\begin{align*}
1 &= \int_I \mathbb{Q}(A_u)\lambda(\mathrm{d}u) \\
  &= \int_H \mathbb{Q}(A_u)\lambda(\mathrm{d}u) + \int_{H^c}\mathbb{Q}(A_u)\lambda(\mathrm{d}u)\\ 
  &\leq c\lambda(H) + \lambda(H^c).
\end{align*}}
Thus $(c-1)\lambda(H)\geq 0$, which implies $\lambda(H)=0$. Therefore
$\mathbb{Q}(A_u)=1$ for $\lambda$-a.e.\ $u$, i.e., the paths
$t\mapsto B_t^u(\omega)$ are continuous for $\mathbb{Q}$-a.e.\ $\omega$.

Hence  $(B^u)_{u \in I}$ is an e.p.i family of
Brownian motion.

\end{proof}

{ \begin{remark}
We emphasize that the theorem establishes an equivalence between an e.p.i. family of Brownian motions on the sample space and a Brownian motion on the Fubini extension with identically distributed e.p.i. elements. The condition that the family be identically distributed is necessary. Indeed, without it, one can construct an e.p.i. family that aggregates to a Brownian motion on the extended space, while the individual processes are not Brownian motions themselves.

To illustrate this, let $\mathsf{W}_{\mathbb{R}}:=C_{\mathbb{R}}(\mathbb{R}_{+})$ be equipped with the topology of uniform convergence on every compact subset of $\mathbb{R}_{+}$ and let $\mu_B$ be the Wiener measure on $\mathsf{W}_{\mathbb{R}}$ with canonical process $(B_t)_{t \geq 0}$. Consider the set
{{\small} \begin{align*}
A := \{x \in \mathsf{W}_{\mathbb{R}}: B_1(x) \geq 0\},
\end{align*}}
so that $\mu_B(A)=1/2$, and define the conditional measures $\mu^+ := \mu_B(\cdot\mid A)$ and $\mu^- := \mu_B(\cdot\mid A^c)$. Then $\mu_B = \tfrac{1}{2}\mu^+ + \frac{1}{2}\mu^-$. Under $\mu^+$ or $\mu^-$, the canonical process is not a Brownian motion.

Now, work on the Fubini extension $(I\times\Omega,\mathcal{V},\mathcal{Q})$ of the product space $(I\times\Omega, \mathcal{F}\otimes \mathcal{I}, \lambda \otimes \mathbb{Q})$. By Theorem~\ref{thm-sun}, there exists a $\mathcal{V}$-measurable map $\Theta:I\times\Omega\to \mathsf{W}_{\mathbb{R}}$ such that the family $\{\Theta(u,\cdot):u\in I\}$ is e.p.i. and
{{\small} \begin{align*}
\mathbb{Q}\circ\Theta(u,\cdot)^{-1} =
\begin{cases}
\mu^+,& u<1/2,\\
\mu^-,& u \geq 1/2.
\end{cases}
\end{align*}}
For each $u\in I$ and $t \geq 0$, set $X_t^u(\omega) := \Theta(u,\omega)(t)$ and $\mathbf{X}_t(u,\omega) := X_t^u(\omega)$. Then, $(X_t^u)_{t\ge 0}$ has law $\mu^+$ or $\mu^-$ and is therefore not a Brownian motion. On the other hand, for any $0\le t_1<\dots<t_k$ and any bounded measurable $\psi:\mathbb{R}^k\to\mathbb{R}$, Fubini's theorem and the identity $\mu_B=\frac12\mu^+ + \frac12\mu^-$ yield
{{\small} \begin{align*}
\mathbb{E}_{\mathcal{Q}}\big[\psi(\mathbf{X}_{t_1},\dots,\mathbf{X}_{t_k})\big]
&= \int_I \mathbb{E}_{\mathbb{Q}}\big[\psi(X_{t_1}^u,\dots,X_{t_k}^u)\big]\lambda(\mathrm{d}u)\\
&= \frac{1}{2}\int_{C_{\mathbb{R}}} \psi(x(t_1),\dots,x(t_k))\mu^+( \mathrm{d}x)
   + \frac12 \int_{C_{\mathbb{R}}} \psi(x(t_1),\dots,x(t_k))\mu^-( \mathrm{d}x)\\
&= \int_{C_{\mathbb{R}}} \psi(x(t_1),\dots,x(t_k))\mu_B(\mathrm{d}x).
\end{align*}}
Thus, $(\mathbf{X}_t)_{t \geq 0}$ has the same finite-dimensional distributions as a standard Brownian motion and continuous paths by construction. Hence $\mathbf{X}$ is a Brownian motion on $(I\times\Omega,\mathcal{V},\mathcal{Q})$.
\end{remark}}

We conclude this section with an important result on the Fubini extension which is the Exact Law of Large Numbers (ELLN).

\begin{proposition}[\cite{sun06exact} Corollary 2.10 ]\label{ELLN}
Let $\Theta $ be an integrable random variable on $(I\times \Omega,\mathcal{V},\mathcal{Q})$. If $\Theta$ is e.p.i., then the sample mean equals the mean of $\Theta$, 
    \begin{align*}
\int_{I}\Theta(u,\tilde{\omega})\lambda(\mathrm{d}u) = \int_{I \times \Omega} \Theta(u,\omega)\mathcal{Q}(\mathrm{d}u,\mathrm{d}\omega), \;\; \forall \tilde{\omega} \in \Omega. 
    \end{align*}
\end{proposition}

In particular for  a Brownian motion $\mathbf{B}$  on the Fubini extension space $(I\times \Omega,\mathcal{V},\mathcal{Q})$ with e.p.i elements, for all $(t,\omega) \in [0,T]\times\Omega,$
\begin{align*}
    \int_{I}B^{u}_t (\omega)\lambda(\mathrm{d}u) &= \int_{I} \mathbb{E}_{\mathbb{{Q}}}[B^u_t]\lambda(\mathrm{d}u)= \mathbb{E}_{\mathcal{Q}}[\mathbf{B}_t]= 0.
\end{align*}

\section{Single equation formulation}\label{sec:application}

Thanks to Theorem~\ref{thmwienerfubini} and the notion of the Fubini extension space, we can directly identify a system of stochastic differential equations, known in the literature as graphon SDEs~\cite{crucianelli24interacting} as a single equation on the Fubini extension space driven by a standard Brownian motion. 

In order to keep the presentation simple and avoid unnecessary technicalities, we work in the following setting: 
{{\small}
\begin{align*}
\mathcal{H} &:= L^2_{\mathbb{R}}(I\times\Omega,\mathcal{Q}), \quad \|X\|_{\mathcal{H}}^{2} := \int_{I\times\Omega}|X(u,\omega)|^2\mathcal{Q}(\mathrm{d}u,\mathrm{d}\omega),\\
\mathcal{H}_I &:= L^2_{\mathbb{R}}(I,\lambda), \quad \|\xi\|_{\mathcal{H}_I}^{2} := \int_{I}|\xi(u)|^2\lambda(\mathrm{d}u).
\end{align*}
}

For simplicity of presentation, we impose the following classical
assumptions on the coefficients of the SDEs. While more general conditions could be considered, the present formulation
is sufficient for our purposes.

\begin{assumption}[Lipschitz and linear growth conditions]\label{assumption1}
The functions $b,\sigma : \mathbb{R}\times\mathbb{R} \mapsto \mathbb{R}$ satisfy:
{{\small}
\begin{align}
    \|b(x,m) - b(x',m')\|_2 
      &\leq C\Big( \|x-x'\|_{2} + \|m-m'\|_2 \Big),\nonumber\\
    \|\sigma(x) - \sigma(x')\|_{2} &\leq C \|x-x'\|_{2} \\
    \|b(x,m)\|_{2} &\leq C\big(1 + \|x\|_{2} + \|m\|_{2}\big),\nonumber \\
    \|\sigma(x)\|_{2} &\leq C\big(1 + \|x\|_{2}\big),
\end{align}}
for some constant $C>0$ and for all $(x,m),(x',m')\in\mathbb{R}^2$.
\end{assumption}

\subsection{Graphon stochastic differential equations}
{ We briefly recall the notion of graphons relevant to our setting
(see \cite{lovasz12large}). A \emph{graphon} is a symmetric and Borel measurable function $ w : I\times I \to I $. We  
denote by $\mathcal{W}$ the space of graphon function.

Intuitively, graphons generalize the notion of graphs when we have a continuum of vertices, where $w(u,v)$ would represent the weight of edge between $u $ and $v$. These objects are particularly well-suited for describing limits
of convergent sequences of dense graphs.  

In recent years, a growing literature \cite{aurell22sto,bayraktar23graphon,coppini25nonlinear,crucianelli24interacting,amini25gqfs} deals with  a system of stochastic differential equations with graphon interaction, called graphon SDEs, of the form 

{\begin{align}\label{system}
\mathrm{d}\theta_{t,u}
  &= b\left(\theta_{t,u}, \int_I w(u,v)\phi_{t,v}\lambda(\mathrm{d}v)\right)\mathrm{d}t
    + \sigma\left(\theta_{t,u}\right)\mathrm{d}B_t^{u}\\
\phi_{t,u} &:= \mathbb{E}_{\mathbb{Q}}\left[\theta_{t,u}\right]\nonumber,
\end{align}}
where $w \in \mathcal{W}$ is a graphon function, and the collection $ \mathbf{B} = \{ B^{u} : u \in I \}$ is a family of e.p.i. Brownian motions.

\begin{remark}
In the above cited literature, one often encounters more general formulations of
graphon SDEs. For instance, in \cite{crucianelli24interacting} the dynamics are
typically written as
{{\small}
\begin{align*}
\mathrm{d}\theta_{t,u}
&= b\Big(t,\theta_{t,u},\int_{I} w(u,v)\,\mathcal{L}(\theta_{t,v})\,\lambda(\mathrm{d}v)\Big)\mathrm{d}t
 + \sigma\Big(t,\theta_{t,u},\int_{I} w(u,v)\,\mathcal{L}(\theta_{t,v})\,\lambda(\mathrm{d}v)\Big)\mathrm{d}B_t^{u},
\end{align*}
}
where the interaction kernel $w : I\times I \to \mathbb{R}$ is not necessarily
bounded and the measure $\mathcal{L}(\theta_{t,v})$ denotes the law of $\theta_{t,v}$. Our presentation restricts to the simpler case
\eqref{system} in order to streamline the exposition, but all the arguments
below extend without difficulty to such more general settings.
\end{remark}
 
Thanks to Theorem~\ref{thmwienerfubini}, the graphon SDE system \eqref{system} can be
expressed on the Fubini extension as a single SDE driven by a standard Brownian motion
$\mathbf{B} = (\mathbf{B}_t)_{t\ge 0}$. Define the process
\[
\Theta_t(u,\omega) := \theta_{t,u}(\omega), \qquad (u,\omega)\in I\times\Omega,
\]
and set $\Theta := (\Theta_t)_{t\ge 0}$. Then $\Theta$ solves
\begin{align}\label{compactform}
    \mathrm{d}\Theta_{t}
    &= b\big(\Theta_t,\mathbb{W}[\Phi_t]\big)\,\mathrm{d}t
      + \sigma(\Theta_t)\,\mathrm{d}\mathbf{B}_t,
\end{align}
where $\Phi_t : I\to\mathbb{R}$ is given by
\[
\Phi_t(u) := \mathbb{E}_{\mathbb{Q}}\big[\Theta_t(u,\cdot)\big],
\]
and $\mathbb{W}$ is the interaction operator
\[
\mathbb{W} : \mathcal{H}_I \to \mathcal{H}_I, \qquad
\eta \mapsto \mathbb{W}[\eta],
\]
defined for all $u\in I$ by
\[
\mathbb{W}[\eta](u) := \int_I w(u,v)\,\eta(v)\,\lambda(\mathrm{d}v).
\]

\begin{remark}
The operator $\mathbb{W}$ is well-defined and bounded on $\mathcal{H}_I$. We have by the Cauchy-Schwarz inequality,
{{\small}
\begin{align*}
    |\mathbb{W}[\eta](u)|^2 
    &= \Big| \int_{I} w(u,v)\eta(v)\lambda(\mathrm{d}v) \Big|^2 \\
    &\leq \left( \int_{I} |w(u,v)|^2 \lambda(\mathrm{d}v) \right) \left( \int_{I} |\eta(v)|^2 \lambda(\mathrm{d}v) \right) \\
    &\leq  \|\eta\|_{\mathcal{H}_I}^2.
\end{align*}}
Integrating over $u \in I$, we obtain $\|\mathbb{W}[\eta]\|_{\mathcal{H}_I} \leq  \|\eta\|_{\mathcal{H}_I}$.
Furthermore, $\mathbb{W}$ is a Hilbert-Schmidt operator since the graphon $w$ is bounded and thus square-integrable on $I\times I$.
\end{remark}
\begin{remark}
It is worth noting that  Equation \ref{compactform} does not coincide completely with
the standard McKean-Vlasov formulation, due to the presence of the kernel $\mathbb{W}$ and the expectation taken under
$\mathbb{Q}$. Our formulation also differs from the $L^2$--formulation in
\cite[Equation~3.3]{coppini25nonlinear}, where the stochastic system is
written as an equation in $\mathcal{H}_{I}$. In contrast, Equation~\ref{compactform}
treats $\Theta$ directly as a real-valued stochastic process on the Fubini extension space. A formulation close to ours was given in \cite[Equation~7]{aurell22sto}. In the special case of a constant graphon $w \equiv c$, we recover the standard McKean-Vlasov equation. Indeed, $\mathbb{W}[\mathbb{E}_{\mathbb{Q}}[\Theta_t]] = c\int_{I}\mathbb{E}_{\mathbb{Q}}[\Theta_t(u,.)]\lambda(\mathrm{d}u) = c\mathbb{E}_{\mathcal{Q}}[\Theta_t]$, and Equation~\ref{compactform} reduces to 
{{\small}
\begin{align*}
    \mathrm{d}\Theta_t &= b(\Theta_t,c\mathbb{E}_{\mathcal{Q}}[\Theta_t])\mathrm{d}t + \sigma(\Theta_t)\mathrm{d}\mathbf{B}_t.
\end{align*}}
\end{remark}

We now address the well-posedness of the single equation formulation \eqref{compactform}.

{\begin{proposition}\label{prop:well}
Under Assumption \ref{assumption1} and if $\Theta_0 \in \mathcal{H}$, there exists a unique solution $\Theta := (\Theta_t)_{0 \leq t \leq T} $ satisfying Equation~\ref{compactform}.

\end{proposition}}

The proof of this proposition relies on the following general well-posedness
result for a class of weighted McKean-Vlasov SDEs.

\begin{lemma}\label{wellposed-lemma} 
Under Assumption~\ref{assumption1}, consider a probability space $(\mathscr{A}, \mathcal{G}, \mathrm{P})$
    carrying a standard Brownian motion $(W_t)_{t \geq 0}$. Let $k : \mathscr{A} \times \mathscr{A} \longrightarrow \mathbb{R}$ be a bounded measurable kernel, and consider the stochastic differential equation~\footnote{Let us emphasize that Equation~\ref{weighted-mckean-vlasov} is not the standard McKean-Vlasov SDE because of
the presence of the kernel $k$.
     When the kernel is constant in the two-variables, we recover the simplest  Mckean-Vlasov equation.}
    {{\small} \begin{align}\label{weighted-mckean-vlasov}
    \mathrm{d}X_t = b(X_t, Y_t)\mathrm{d}t + \sigma(X_t)\mathrm{d}W_t,\qquad
    X_0 \in L^2_{\mathbb{R}}(\mathscr{A})x,
    \end{align}}
    where the coupling term $Y_t$ is defined as the weighted average
    {{\small} \begin{align*}
    Y_t(a) &= \int_{\mathscr{A}} k(a,{a'}) X_t(a')\mathrm{P}(\mathrm{d}a'), \; \forall a \in \mathscr{A}.
    \end{align*}}
    {Then \eqref{weighted-mckean-vlasov} admits a unique strong solution.}
\end{lemma}

\begin{proof}
The coupling term $Y_{\cdot}$ is well defined if  $X_{\cdot} \in L^2_{\mathbb{R}}(\mathscr{A})$. Indeed, 
{{\small} \begin{align*}
    \|Y_{\cdot}\|_{2}^{2} &= \int_{\mathscr{A}}\sup_{ 0 \leq t \leq T} |Y_{t}(a)|^{2}\mathrm{P}(\mathrm{d}a)\\
    &= \int_{\mathscr{A}}\sup_{0 \leq t \leq T}|k(a,a')X_t(a')|^2\mathrm{P}(\mathrm{d}a')\\
&\leq\|k\|_{\infty}\int_{\mathscr{A}}\sup_{0 \leq t \leq T}|X_t(a')|^2\mathrm{P}(\mathrm{d}a')\\
    &\leq \|k\|_{\infty}\|X_{\cdot}\|_2^{2} < \infty.
\end{align*}}
For any $\sigma(W_t)$-adapted $Z_{\cdot} \in L^2_{\mathbb{R}}(\mathscr{A})$, we introduce the following nonlinear map:
{{\small}
\begin{align*}
    (\Xi[Z_{\cdot}])_{t} &= X_0 + \int_{0}^{t}b(Z_s,\tilde{Z}_s)\mathrm{d}s + \int_{0}^{t}\sigma(Z_s)\mathrm{d}W_s,
\end{align*}}
where $\tilde{Z}_t(a) = \int_{\mathscr{A}}k(a,a')Z_t(a')\mathrm{P}(\mathrm{d}a') $. By Assumption~\ref{assumption1}, we obtain directly
{{\small} \begin{align*}
    \|\Xi[Z_{\cdot}]\|_{2}^{2} &\leq C\Big( 1 + \|Z_{\cdot}\|^2_2\Big). 
\end{align*}}

Hence $\|\Xi[Z_{\cdot}]\|_{2} < \infty$ for all $\sigma(W_t)$-adapted
$Z_{\cdot} \in L^2_{\mathbb{R}}(\mathscr{A})$, and $\Xi[Z_{\cdot}]$ is again
$\sigma(W_t)$-adapted. The continuity of the mapping $\Xi$ follows directly
from the Lipschitz assumption in \eqref{assumption1}. Now using the Picard iteration method, we show that the mapping $\Xi$ admits
a unique fixed point. Starting from any $\sigma(W_t)$-adapted process
$Z^{0}_t \in L^2_{\mathbb{R}}(\mathscr{A})$, we consider the sequence
$Z^{n+1}_{\cdot} = \Xi[Z_{\cdot}^{n}]$, noting that 
{{\small} 
\begin{align*}
    \|(\Xi(Z_{\cdot})_t - (\Xi[Z'_{\cdot}])_t\|_2 &\leq \sqrt{t}\|b(Z_{\cdot},\tilde{Z}_{\cdot}) - b(Z'_{\cdot},\tilde{Z}'_{\cdot})\|_2 + \|\sigma(Z_{\cdot}) - \sigma(Z_{\cdot}')\|_2\\
    &\stackrel{\text{Lipschitz}}{\leq} 2C(\sqrt{T}+1)\|Z_{\cdot} - Z_{\cdot}'\|,
\end{align*}}
setting $K = 2C(\sqrt{T}+1)$ and iterating the bound, we obtain 
{{\small} 
\begin{align*}
    \|(\Xi^{(n)}(Z_{\cdot})_t^{} - (\Xi^{(n)}[Z'_{\cdot}])_t\|_2^2 &\leq \frac{K^{2n}T^n}{n!}\|Z_{\cdot} -Z_{\cdot}'\|^2_{2}.
\end{align*}}
For $n$ large enough, this defines a contraction; since the space is complete,
the sequence satisfies
{\begin{align*}
\Xi^{(n)}[Z_{\cdot}] = Z_{\cdot}^{n+1} \longrightarrow Z_{\cdot},
\end{align*}}
and the limit $Z_{\cdot}$ is a fixed point of $\Xi$, which conclude existence.

For uniqueness, let $X$ and $X'$ be two strong solutions to \eqref{weighted-mckean-vlasov} starting from the same initial condition $X_0$. Since $X$ and $X'$ are fixed points of the map $\Xi$,  we can apply the inequality established above:
{{\small}
\begin{align*}
    \|X_{\cdot} - X'_{\cdot}\|_{2}^2 &= \|\Xi^{(n)}[X] - \Xi^{(n)}[X']\|_{2}^2 \\
    &\leq \frac{K^{2n}T^n}{k!}  \|X_{\cdot} - X'_{\cdot}\|_{2}^2 \mathrm{d}s.
\end{align*}}
Letting $n \to \infty$, we find that $\|X_{\cdot} - X'_{\cdot}\|_{2}^2 = 0$, so
$X_{\cdot} = X'_{\cdot}$ $\mathrm{P}$-almost surely.
\end{proof}

We now show that the compact formulation \eqref{compactform} fits into the
framework of Lemma~\ref{wellposed-lemma}.

\begin{proof}[Proof of Proposition~\ref{prop:well}]
The statement follows directly from Lemma~\ref{wellposed-lemma}. Indeed, Equation~\ref{compactform}  is a particular case of  \eqref{weighted-mckean-vlasov} with the probability space $(I\times\Omega,\mathcal{V},\mathcal{Q}) = (\mathscr{A},\mathcal{G},\mathrm{P})$, driving Brownian motion $W_t = \mathbf{B}_t$, and kernel $k$ corresponding to
the graphon $w$. More precisely, for $(u,\omega)\in I\times\Omega$,
{{\small} \begin{align*}
    \mathbb{W}[\Phi_t](u,\omega) &= \mathbb{W}[\mathbb{E}_{\mathbb{Q}}[\Theta_t]](u,\omega)\\
    &= \int_{I}w(u,v)\mathbb{E}_{\mathbb{Q}}[\Theta_t]\lambda(\mathrm{d}v)\\
    &\stackrel{\text{Fubini}}{=} \int_{I\times\Omega}w(u,v)\Theta_t(v,\omega)\mathcal{Q}(\mathrm{d}v,\mathrm{d}\omega),
\end{align*}}
and
Lemma~\ref{wellposed-lemma} yields existence and uniqueness of its strong
solution.
\end{proof}
}

\subsection{Spectral representation and Monte Carlo approximation}

In this subsection we show how the compact formulation~\eqref{compactform}
can be analyzed via spectral representations, starting from the case of a separable kernel.

Consider now the interacting case with separable kernel $w(u,v) = \phi(u)\phi(v)$, where $\phi \in L^2_{I}(I)$:
 {{\small}
    \begin{align*}
        \mathrm{d}\theta_{t,u} &= b\Big(\theta_{t,u} , \int_{I}\phi(u)\phi(v)\mathbb{E}\big[\theta_{t,v}\big]\lambda(\mathrm{d}v)\Big)\mathrm{d}t + \sigma(\theta_{t,u})\mathrm{d}B_t^{u}.
    \end{align*}}

To lift this to the Fubini extension space, let $U$ be the random variable on $(I,\mathcal{I},\lambda)$ defined by the identity map $U(u) = u$, and view it on
$I\times\Omega$ as $U(u,\omega)=u$. Then the interaction term can be rewritten
using the expectation on the Fubini space $\mathcal{Q}$:
    $$
        \int_{I}\phi(u)\phi(v)\mathbb{E}_{\mathbb{Q}}\big[\theta_{t,v}\big]\lambda(\mathrm{d}v) = \phi(u) \int_{I} \phi(v)\mathbb{E}_{\mathbb{Q}}\big[\theta_{t,v}\big]\lambda(\mathrm{d}v) = \phi(u)\mathbb{E}_{\mathcal{Q}}[\phi(U)\Theta_t].
    $$
    Consequently, the single process $\Theta$ on the Fubini extension satisfies
    {{\small}
    \begin{align*}
        \mathrm{d}\Theta_t = b\big(\Theta_t , \; \phi(U)\mathbb{E}_{\mathcal{Q}}[\phi(U)\Theta_t]\big)\mathrm{d}t + \sigma(\Theta_{t})\mathrm{d}\mathbf{B}_t.
    \end{align*}}

The separable case is of particular interest as it facilitates numerical simulation via particle systems and Monte-Carlo methods, where indices are drawn uniformly at random. As we explain below, this intuition extends beyond
separable kernels to more general graphons through spectral decomposition.

Since the graphon $w \in L^2(I^2)$ is symmetric, by the spectral theorem, it admits the decomposition
$$w(u,v) = \sum_{i=1}^{\infty} \alpha_i \phi_i(u)\phi_i(v), $$
where $\{\alpha_i\}_{i \geq 1}$ are the (real) eigenvalues and $\{\phi_i\}_{i \ge 1}$ form an orthonormal basis of $L^2(I)$.

Let $w^d$ denote the truncated graphon $w^d(u,v) := \sum_{i=1}^{d} \alpha_i \,\phi_i(u)\phi_i(v)$, and let $\mathbb{W}^d$ be the associated integral operator, defined for any function $\eta \in L^2(I)$ by
$$ \mathbb{W}^d[\eta](u) := \int_I w^d(u,v)\eta(v)\lambda(\mathrm{d}v) = \sum_{i=1}^{d}\alpha_i \phi_i(u) \int_{I}\phi_i(v)\eta(v)\lambda(\mathrm{d}v). $$

We consider the truncated process $\Theta^d$ on the Fubini extension $(I\times\Omega, \mathcal{V}, \mathcal{Q})$ governed by
\begin{align*}
    \mathrm{d}\Theta_t^d &= b\big(\Theta_t^d, \mathbb{W}^d[\Phi_t^d]\big)\mathrm{d}t + \sigma(\Theta_t^d)\mathrm{d}\mathbf{B}_t,\\
     &= b\Bigl(\Theta_t^d,\sum_{i=1}^{d}\alpha_{i}\phi_{i}(U)\mathbb{E}_{\mathcal{Q}}\bigl[\phi_{i}(U)\Theta_t^d\bigr]\Bigr)\mathrm{d}t + \sigma(\Theta_t^d)\mathrm{d}\mathbf{B}_t,
\end{align*}
where $\Phi_t^d(u) := \mathbb{E}_{\mathbb{Q}}[\Theta_t^d(u,\cdot)]$ and
$U(u,\omega)=u$.

\begin{proposition}\label{convergence}
Let $T>0$. Under Assumption~\ref{assumption1} and assuming the initial condition satisfies $\mathbb{E}_{\mathcal{Q}}[|\Theta_0|^2] < \infty$ , there exists a constant $C_T > 0$ such that
$$ \sup_{t \in [0,T]} \|\Theta_t - \Theta_{t}^d\|_{\mathcal{H}}^2 \leq C_T \, \|w - w^d\|_{2}^2. $$
\end{proposition}

\begin{proof}
By It\^o's formula,
{\small}\begin{align*}
    \frac{\mathrm{d}}{\mathrm{d}t}\mathbb{E}_{\mathcal{Q}}\big[\big|\Theta_{t} - \Theta_{t}^{d}\big|^2 \big] &\leq 2\mathbb{E}_{\mathcal{Q}}\Big[\int_{0}^{t}\big|(\Theta_{s} - \Theta_{s}^d)(b(\Theta_s,\mathbb{W}[\Phi_s]) - b(\Theta_s^{d},\mathbb{W}^{d}[\Phi_s^{d}]) \big|\mathrm{d}s + \int_{0}^{t}|\sigma(\Theta_s) - \sigma(\Theta_{s}^d)|^2\mathrm{d}s\Big]\\
    &\stackrel{\text{Lipschitz}}\leq (2C +C^2 + 1)\mathbb{E}_{\mathcal{Q}}[|\Theta_t- \Theta_{t}^d|^2] + C^2\mathbb{E}_{\mathcal{Q}}\big[\big|\mathbb{W}[\Phi_t]- \mathbb{W}^{d}[\Phi_t^d]\big|^2\big]
\end{align*}
Moreover,
{\small} \begin{align*}
    \big|\mathbb{W}[\Phi_t] - \mathbb{W}^{d}[\Phi_t^d]\big|^2 \leq 2\big|(\mathbb{W} - \mathbb{W}^d)[\Phi_t]\big|^2 + 2\big|\mathbb{W}^d[\Phi_t  - \Phi_t^{d}]\big|^2.
\end{align*}
For the first term, using Fubini and Cauchy–Schwarz,
{\small} \begin{align*}
    \mathbb{E}_{\mathcal{Q}}\Big[\big|(\mathbb{W} - \mathbb{W}^d)[\Phi_t]\big|^2\Big] &= \int_{I\times\Omega}\Big|\int_{I}\big(w(u,v)- w^d(u,v)\big)\Phi_t(v)\lambda(\mathrm{d}v)\Big|^2\mathcal{Q}(\mathrm{d}u,\mathrm{d}\omega)\\
    &= \int_{I}\Big|\int_{I}\big(w(u,v)- w^d(u,v)\big)\Phi_t(v)\lambda(\mathrm{d}v)\Big|^2\lambda(\mathrm{d}u)\\
    &\stackrel{\text{C-S}}\leq \Big(\int_{I^2}|w(u,v) - w^d(u,v)|^2\lambda(\mathrm{d}v)\lambda(\mathrm{d}u)\Big)\Big(\int_{I}|\Phi_t(v)|^2\lambda(\mathrm{d}v)\Big)\\
    &\stackrel{\Phi_t(v) = \mathbb{E}_{\mathbb{Q}}[\Theta_t(v,\cdot)] }\leq \|w- w^d\|_{2}^{2}\|\Theta_t\|_{\mathcal{H}}^{2}.
\end{align*}
Since $\mathbb{E}_{\mathcal{Q}}[|\Theta_0|^2] < \infty$, standard SDE theory ensures that the second moment is bounded, i.e., $\|\Theta_t\|_{\mathcal{H}} \leq C$ for all $t\in[0,T]$. For the second term, using the operator norm of $\mathbb{W}^d$ and Jensen,
{\small} \begin{align*}
\mathbb{E}_{\mathcal{Q}}\Big[\big|\mathbb{W}^d[\Phi_t  - \Phi_t^{d}]\big|^2\Big] &\le \|\mathbb{W}\|_{op}^2 \|\Phi_t - \Phi_t^d\|_{\mathcal{H}_I}^2 \\
    &= \|\mathbb{W}\|_{op}^2 \int_I \big| \mathbb{E}_{\mathcal{Q}}[\Theta_t(u,\cdot) - \Theta_t^d(u,\cdot)] \big|^2 \lambda(\mathrm{d}u) \\
    &\le \|\mathbb{W}\|_{op}^2 \mathbb{E}_{\mathcal{Q}}\big[ |\Theta_t - \Theta_t^d|^2 \big].
\end{align*}
Substituting these estimates into the differential inequality, we obtain
$$ \frac{\mathrm{d}}{\mathrm{d}t} \mathbb{E}_{\mathcal{Q}}[|\Theta_t - \Theta_t^d|^2] \leq C \mathbb{E}_{\mathcal{Q}}[|\Theta_t - \Theta_t^d|^2] + C\|w - w^d\|_{2}^2. $$
By Grönwall's lemma, for all $t \in [0,T]$,
$$ \mathbb{E}_{\mathcal{Q}}[|\Theta_t - \Theta_t^d|^2] \le \left( C \|w - w^d\|_{2}^2 \cdot T \right) e^{CT}. $$
The result follows by setting $C_T = C T e^{CT}$.

\end{proof}

Thanks to the single equation formulation, Proposition \ref{convergence}, and standard propagation of chaos results, the graphon SDE system \eqref{system} can be simulated using a Monte Carlo particle system. This approach effectively avoids the need for a deterministic grid discretization of $[0,1]$.

\begin{remark}
One might wonder whether the particle method is justified given the random variable $U$. The independence between $U$ and $\mathbf{B}_t$ holds on the Fubini extension. Indeed, for any bounded measurable functions $f,g$,
{\small} \begin{align*}
\mathbb{E}_{\mathcal{Q}}[f(U)g(\mathbf{B}_t)] &= \int_{I}f(u)\mathbb{E}_{\mathbb{Q}}[g(B_t^u)]\lambda(\mathrm{d}u) \\
&= \mathbb{E}_{\mathbb{Q}}[g(B_t^1)] \int_{I}f(u)\lambda(\mathrm{d}u) 
= \mathbb{E}_{\mathcal{Q}}[f(U)]\mathbb{E}_{\mathcal{Q}}[g(\mathbf{B}_t)].
\end{align*}
This independence ensures that sampling indices $u_k$ (i.i.d.) uniformly provides an unbiased approximation of the interaction terms.
\end{remark}


\begin{algorithm}[H]\label{algoMC}
\caption{Monte Carlo Euler-Maruyama scheme for the truncated graphon SDE}

    \SetAlgoLined
    \DontPrintSemicolon
        \KwData{Number of particles $N$, truncation level $d$, horizon $T$, time step $\Delta t$}
    \KwResult{Trajectories $(\Theta_{n,k}^d)_{n=0,\dots,\lfloor T/\Delta t\rfloor;\,k=1,\dots,N}$}
    
    \BlankLine
    \tcp{1. Initialization}
    Sample i.i.d. indices $u_k \sim \mathcal{U}(0,1)$ for $k=1,\dots,N$\;
    Initialize particle states $\Theta_{0,k}^d$ for $k=1,\dots,N$\;
    
    \BlankLine
    \tcp{2. Time Loop (Euler-Maruyama)}
    \While{$n\Delta t < T$}{
        \For{$k \leftarrow 1$ \KwTo $N$}{
            \tcp{Compute interaction approximation}
            $ \hat{\mu}_{n,i} \leftarrow \frac{1}{N}\sum_{l=1}^{N}\phi_i(u_l)\Theta_{n,l}^d \quad \text{for } i=1\dots d$\;
            
            \tcp{Update step}
            Draw $Z \sim \mathcal{N}(0,1)$\;
            $ \text{Drift} \leftarrow b\Bigg(\Theta_{n,k}^d, \sum_{i=1}^{d}\alpha_{i}\phi_{i}(u_k) \hat{\mu}_{n,i}\Bigg)$\;
            
            $\Theta_{n+1,k}^d \leftarrow \Theta_{n,k}^d + \text{Drift}\cdot\Delta t  + \sigma(\Theta_{n,k}^{d})\sqrt{\Delta t}Z$\;
        }
        $n \leftarrow n+1$\;
    }
\end{algorithm}

Under suitable regularity assumptions, one expects convergence of Algorithm~\ref{algoMC}
as $N\to\infty$ and $d\to\infty$, in line with Proposition~\ref{convergence}
and standard propagation of chaos results.

\section{Girsanov transformation}\label{sec:girsanov}
{
Establishing the existence of a strong solution for certain SDEs can be difficult
outside the standard framework of hypotheses. However, one can often ensure the existence of a weak solution by identifying a probability space on which such a solution can be constructed. A classical method consists in starting from an equation on an auxiliary probability space where well-posedness is easily established. Applying Itô's formula then yields a formal solution to the  desired equation on this auxiliary space; however, the driving noise is generally not a Brownian motion. Through a change of probability measure, via Girsanov's theorem, this noise  can be turned into a Brownian motion.

In the following, we establish a version of Girsanov's theorem for a family of e.p.i. Brownian motions, which yields a standard Brownian motion on a new probability space. This approach is particularly well-suited for establishing weak solutions for graphon SDEs. The advantage of working with a collection of e.p.i. Brownian motions lies in avoiding the need to define martingales on the Fubini extension space; instead, the transformation is performed index by index.

}

\begin{theorem}[Girsanov theorem on a Fubini extension]\label{girsanovlemma}
Let $(I \times \Omega, \mathcal{V}, \mathcal{Q})
$ be  a Fubini extension probability space. Consider a collection of e.p.i. Brownian motions $\mathbf{B} = \big\{B^{u} : u \in I \big\}$, and  a real valued process  $(\Theta_t)_{t \geq 0} = \big\{(\theta_{t,u}^{})_{t \geq 0} : u \in I\big\}$, 
such that for all $(t,u) \in [0,T]\times I$, $\theta_{t,u} $ is $\sigma(B_s^{u}; s\leq t)$-mesurable.

Define the process $(\mathcal{E}_{t})_{t \geq 0} = \Big\{(\mathcal{E}_{t,u})_{t \geq 0} : u \in I\Big\}$ by 
{{\small} \begin{align*} \mathcal{E}_{t} := \exp\Big\{\int_{0}^{t}\Theta_{s}\mathrm{d}\mathbf{B}_s^{} - \frac{1}{2}\int_{0}^{t}\Theta_{s}^{2}\mathrm{d}s \Big\}.
\end{align*}}
If for all $u \in I$, the process $(\mathcal{E}_{t,u})_{t \geq 0}$ is a $\mathbb{Q}$-martingale, then  under the probability measure $\mathcal{P}$ defined by the density  $\mathcal{E}_T$ with respect to $\mathcal{Q}$, i.e.,
$\frac{\mathrm{d}\mathcal{P}}{\mathrm{d}\mathcal{Q}} = \mathcal{E}_T$, 
the process $\mathbf{W} = \big\{W^{u} : u \in I\big\}$ is a standard Brownian motion on the probability space $(I\times \Omega, \mathcal{V}, \mathcal{P}),$ where 
{{\small} \begin{align*}
\mathbf{W}_t &:= \mathbf{B}_t - \int_{0}^{t} \Theta_{s}\mathrm{d}s.
\end{align*}}
\end{theorem}

The above result can be applied, for instance, to construct a weak solution to
the graphon Belavkin equation in \cite[Theorem~2]{amini25gqfs}, which is a
graphon system driven by a family of e.p.i.\ operator-valued Brownian motions.

\begin{proof}[Proof of Theorem~\ref{girsanovlemma}]
The measure $\mathcal{P}$ is a probability measure. Indeed,
\begin{align*}
    \mathcal{P}(I\times\Omega ) &= \int_{I \times \Omega} \mathcal{E}_t(u,\omega)\mathcal{Q}(\mathrm{d}u,\mathrm{d}\omega)\\
    &= \int_{I\times \Omega}\mathcal{E}_{t,u}(\omega)\mathbb{Q}(\mathrm{d}\omega)\lambda(\mathrm{d}u)\\
    &= \int_{I}\underbrace{\mathbb{E}_{\mathbb{Q}}[\mathcal{E}_{t,u}]}_{1}\lambda(\mathrm{d}u)= 1.
    \end{align*}

The law of $\mathbf{W}$ is a Wiener measure.  
Fix $u \in I$ and set
$$
\mathcal{P}^{u}(\mathrm{d}\omega) := \mathcal{E}_{T,u}(\omega)\,\mathbb{Q}(\mathrm{d}\omega).
$$

By the standard Girsanov theorem, $(W_{t}^u)$ is a Wiener process under $\mathcal{P}^u$.

Then, for any measurable function $g : \mathbb{R}^n \to \mathbb{R}$,
{{\small} \begin{align*}
\mathbb{E}_{\mathcal{P}}\left[ g\left( \mathbf{W}_{t_1}, \dots, \mathbf{W}_{t_n} \right) \right]
&= \int_{I\times\Omega} g\left( \mathbf{W}_{t_1}(u,\omega), \dots, \mathbf{W}_{t_n}(u,\omega) \right) \mathcal{P}(\mathrm{d}u,\mathrm{d}\omega) \\
&= \int_{I\times\Omega} g\left( \mathbf{W}_{t_1}(u,\omega), \dots, \mathbf{W}_{t_n}(u,\omega) \right) 
\mathcal{E}_T(u,\omega)\mathcal{Q}(\mathrm{d}u,\mathrm{d}\omega) \\
&= \int_{I\times\Omega} g\left( W_{t_1}^u(\omega), \dots, W_{t_n}^u(\omega) \right) 
\mathcal{E}_{T,u}(\omega)\mathbb{Q}(\mathrm{d}\omega)\lambda(\mathrm{d}u) \\
&= \int_I \mathbb{E}_{\mathcal{P}^u}\left[ g\left( W_{t_1}^u, \dots, W_{t_n}^u \right) \right]\lambda(\mathrm{d}u) \\
&= \int_{\mathbb{R}^n} g(x_1,\dots,x_n) \nu_{t_1,\dots,t_n}(x_1,\dots,x_n)\mathrm{d}x_1\cdots\mathrm{d}x_n,
\end{align*}}
where $\nu_{t_1,\dots,t_n}$ is the joint Gaussian density of $(W_{t_1}^u, \dots, W_{t_n}^u)$ under $\mathcal{P}^u$.

For the independence of increments, consider  $0 \leq s<t$, and take two  measurable functions $g,h$. Then,
{{\small} \begin{align*}
\mathbb{E}_{\mathcal{P}}\left[g(\mathbf{W}_t-\mathbf{W}_s)h(\mathbf{W}_s)\right]
&= \int_I \mathbb{E}_{\mathcal{P}^u}\left[g(W_t^u-W_s^u)h(W_s^u)\right]\lambda(\mathrm{d}u) \\
&= \int_I \mathbb{E}_{\mathcal{P}^u}\left[h(W_t^u-W_s^u)\right]\mathbb{E}_{\mathcal{P}^u}\!\left[h(W_s^u)\right]\lambda(\mathrm{d}u) \\
&= \int_I \mathbb{E}_{\mathcal{P}^u}\left[h(W_t^u-W_s^u)\right]\mathbb{E}_{\mathcal{P}}\left[h(\mathbf{W}_s)\right]\lambda(\mathrm{d}u)\\
&= \mathbb{E}_{\mathcal{P}}\left[h(\mathbf{W}_t-\mathbf{W}_s)\right]\mathbb{E}_{\mathcal{P}}\left[h(\mathbf{W}_s)\right],
\end{align*}}
since under $\mathcal{P}^u$ the increments of $W_t^{u}$ are independent. 

\end{proof}

We end the paper with a few remarks concerning the structure and 
the behavior of the shifted process, and a convenient Novikov-type sufficient condition
ensuring applicability of Theorem~\ref{girsanovlemma}.

\begin{remark}
After the change of measure in Theorem~\ref{girsanovlemma}, there is no
guarantee that the resulting probability space
$(I\times\Omega,\mathcal{V},\mathcal{P})$ is itself a Fubini extension of a
product probability space for some index space $(I,\mathcal{I}',\lambda')$ and
sample space $(\Omega,\mathcal{F}',\mathbb{P})$. The construction ensures that
$\mathbf{W}$ is a Brownian motion on $(I\times\Omega,\mathcal{V},\mathcal{P})$,
but it does not, in general, identify $\mathcal{P}$ as an extension of $\lambda'\otimes\mathbb{P}$.
\end{remark}

\begin{remark}
Let $\mathbf{W}$ be as in Theorem~\ref{girsanovlemma}. Then:
\begin{itemize}
\item[(i)] For fixed $u\in I$, the processes $((W_t^{u})_{t \in [0,T]})_{u \in I}$ need not be a
Brownian motion under $\mathcal{P}$. For
any bounded measurable $g$,
{{\small} \begin{align*}
\mathbb{E}_{\mathcal{P}}[g(W_t^{u})] &= \int_{I\times\Omega}g(W_t^{u}(\omega))\mathcal{P}(\mathrm{d}v,\mathrm{d}\omega)\\
&= \int_{I\times\Omega}g(W_t^{u}(\omega))\mathcal{E}_T(v,\omega)\mathcal{Q}(\mathrm{d}v,\mathrm{d}\omega)\\
&= \int_{\Omega}g(W_t^{u}(\omega))\int_{I}\mathcal{E}_T(v,\omega)\lambda(\mathrm{d}v)\mathbb{Q}(\mathrm{d}\omega)\\
&\stackrel{\text{ELLN}}{=} \int_{\Omega}g(W_t^{u}(\omega))\int_{I\times\Omega}\mathcal{E}_T(v,\tilde{\omega})\mathbb{Q}(\mathrm{d}\tilde{\omega})\lambda(\mathrm{d}v)\mathbb{Q}(\mathrm{d}\omega)\\
&\stackrel{\text{Fubini}}{=} \int_{\Omega}g(W_t^{u}(\omega))\int_{I} \underbrace{\mathbb{E}_{\mathbb{Q}}\big[\mathcal{E}_{T,v}\big]}_{=1}\lambda(\mathrm{d}v)\mathbb{Q}(\mathrm{d}\omega)\\
&= \int_{\Omega}g(W_t^{u}(\omega))\mathbb{Q}(\mathrm{d}\omega) = \mathbb{E}_{\mathbb{Q}}[g(W_t^{u})].
\end{align*}}
In particular, for $g(x) = x$,
{{\small} 
\begin{align*}
\mathbb{E}_{\mathcal{P}}[W_t^{u}] = \mathbb{E}_{\mathbb{Q}}[B_t^{u} - \int_{0}^{t}\theta_{t,u}^{}\mathrm{d}s]
= -\int_{0}^{t}\mathbb{E}_{\mathbb{Q}}[\theta_{t,u}^{}]\mathrm{d}s \neq 0.
\end{align*}
}

\item[(ii)] Nevertheless, the family
$\big((W_t^{u})_{t\in[0,T]}\big)_{u\in I}$ remains pairwise independent under
$\mathcal{P}$. For bounded measurable $g,h$ and $u\neq v$,
{{\small} \begin{align*} \mathbb{E}_{\mathcal{P}}[g(W_t^{u})h(W_s^{v})] 
        &= \int_{I\times\Omega} g(W_t^{u}(\omega))h(W_s^{v}(\omega)) \,\mathcal{P}(\mathrm{d}x,\mathrm{d}\omega)\\
        &= \int_{I \times \Omega} g(W_t^{u}(\omega))h(W_s^{v}(\omega)) \,\mathcal{E}_{T}(x,\omega)\mathcal{Q}(\mathrm{d}x,\mathrm{d}\omega)\\
        &= \int_{\Omega}\Big[\int_{I}\mathcal{E}_{T}(x,\omega)\lambda(\mathrm{d}x)\Big] 
           g(W_t^{u}(\omega))h(W_s^{v}(\omega)) \,\mathbb{Q}(\mathrm{d}\omega)\\
        &\stackrel{\text{ELLN}}{=} \int_{\Omega}\Big[\int_{I\times\Omega}\mathcal{E}_{T}(x,\tilde{\omega})\mathcal{Q}(\mathrm{d}x,\mathrm{d}\tilde{\omega})\Big] 
           g(W_t^{u}(\omega))h(W_s^{v}(\omega)) \,\mathbb{Q}(\mathrm{d}\omega)\\
        &= \mathbb{E}_{\mathbb{Q}}[g(W_t^{u})h(W_s^{v})] = \mathbb{E}_{\mathbb{Q}}[g(W_t^{u})]\;\mathbb{E}_{\mathbb{Q}}[h(W_s^{v})]\\
        &= \mathbb{E}_{\mathcal{P}}[g(W_t^{u})]\;\mathbb{E}_{\mathcal{P}}[h(W_s^{v})].
\end{align*}}
Thus $W^u$ and $W^v$ remain independent under $\mathcal{P}$.

\item[(iii)] The processes $((B_t^{u})_{t \in [0,T]})_{u \in I}$ remain Brownian motions under the probability measure $\mathcal{P}$.

\end{itemize}
\end{remark}


\begin{remark}[Novikov condition on a Fubini extension]
The stochastic exponential in Theorem~\ref{girsanovlemma} is well defined under
a natural Novikov-type integrability condition. 
Let $(I \times \Omega, \mathcal{V}, \mathcal{Q})$ be a Fubini extension probability space. Consider a collection of e.p.i. Brownian motions  $\mathbf{B} = \big\{B^{u} : u \in I \big\}$, and  a real valued  process  $(\Theta_t)_{t \geq 0} = \Big\{(\theta_{t,u}^{})_{t \geq 0} : u \in I\Big\}$, 
where $\Theta$ is e.p.i. and $ \Theta \in L^{2}_{\mathcal{C}_{\mathbb{R}}}(I\times\Omega ,\mathcal{Q})$. Furthermore, assume that for all $(t,u) \in I\times [0,T]$, $\theta_{t,u}^{} $ is $B_t^{u}$-measurable. If the condition
\begin{align*}
\mathbb{E}_{\mathcal{Q}}\Big[\exp\Big\{\int_0^T |\Theta_s|^2 \mathrm{d}s\Big\}\Big]
=\int_I \mathbb{E}_{\mathbb{Q}}\Big[\exp\Big\{\int_0^T (\theta_{t,u}^{})^2 \mathrm{d}s\Big\}\Big]\lambda(\mathrm{d}u)<\infty
\end{align*}
is satisfied, then the process $(\mathcal{E}_{t})_{t \geq 0} = \Big\{(\mathcal{E}_{t,u})_{t \geq 0} : u \in I\Big\}$ with 
$$\mathcal{E}_{t} = \exp\Big\{\int_{0}^{t}\Theta_{s}\mathrm{d}\mathbf{B}_s^{} - \frac{1}{2}\int_{0}^{t}\Theta_{s}^{2}\mathrm{d}s \Big\}$$
forms a collection $\Big\{(\mathcal{E}_{t,u})_{t \geq 0} : u \in I\Big\}$ of e.p.i. $\mathbb{Q}$-martingales on the Fubini extension. The proof can be covered by the classical Novikov theorem, see e.g., \cite[Section 8.1]{revuz13continuous}.

\end{remark}

}

\paragraph*{Funding.}
This work was supported by the ANR projects Q-COAST (ANR-19-CE48-0003) and IGNITION (ANR-21-CE47-0015).

\bibliographystyle{plain}
\bibliography{biblio}

\end{document}